\documentclass[reqno,12pt]{article}
\usepackage{amsmath,amsthm,amssymb,bm}
\usepackage{datetime}

\def\cA{{\mathcal A}}

\def\cH{{\mathcal H}}

\def\cJ{{\mathcal J}}

\def\cG{{\mathcal G}}

\newtheorem{theorem}{Theorem}
\newtheorem{proposition}[theorem]{Proposition}
\newtheorem{lemma}[theorem]{Lemma}

\usepackage{tikz}
\usetikzlibrary{shapes.geometric}
\usetikzlibrary{fit}
\usetikzlibrary{calc}

\long\def\symbolfootnote[#1]#2{\begingroup\def\thefootnote{\fnsymbol{footnote}}
\footnote[#1]{#2}\endgroup}

\begin{document}

\title{A note on intersecting hypergraphs with large cover number}
\author{P.E.~Haxell\footnote{Department of Combinatorics and
Optimization, University of Waterloo, Waterloo, Ontario, Canada N2L
3G1. pehaxell@uwaterloo.ca; Partially
supported by NSERC.}\ \ 
and A.D.~Scott\footnote{Mathematical Institute,
    University of Oxford, Andrew Wiles Building, Radcliffe Observatory
    Quarter, Woodstock Rd, Oxford, UK OX2 6GG. scott@maths.ox.ac.uk}}
\date{}

\maketitle

\begin{abstract} 
\noindent
We give a construction of $r$-partite $r$-uniform intersecting
hypergraphs with cover number at least $r-4$ for all but finitely many
$r$. This answers a question of Abu-Khazneh, Bar\'at, Pokrovskiy and
Szab\'o, and shows that a long-standing unsolved conjecture due to Ryser  is close to being best possible for every value of $r$.
\end{abstract}

\noindent {\bf Keywords:} partite hypergraph, intersecting, cover


\section{Introduction}
A hypergraph is said to be $r$-{\it partite} if it has a vertex
partition $V_1\cup\cdots\cup V_r$ such that each edge contains at most
one vertex from each $V_i$. 
An old and well-studied conjecture of Ryser~\cite{R} states that every
$r$-partite $r$-uniform hypergraph $\cH$ satisfies
$\tau(\cH)\le(r-1)\nu(\cH)$, where  
$\tau(\cH)$ denotes the minimum size of a vertex cover and
$\nu(\cH)$ denotes the maximum size of a set of
pairwise disjoint edges in $\cH$. In particular this would imply that
every intersecting $r$-partite $r$-uniform hypergraph can be covered
by $r-1$ vertices. Despite substantial work by many authors over many years,
Ryser's Conjecture is known to be true in general only for $r=2$
(K\"onig's Theorem) and $r=3$~\cite{A}, and for intersecting
hypergraphs only for $r\leq 
5$~\cite{T}. For more on the history of the problem see e.g.~\cite{HS}.

Ryser's Conjecture is tight for a given value of $r$ if there is 
an $r$-partite $r$-uniform hypergraph $\cH$ with 
$\tau(\cH)\ge (r-1)\nu(\cH)$ (such hypergraphs are called $r$-{\em
  Ryser} hypergraphs in~\cite{ABPS}). Because of the apparent difficulty of
the problem in general, a significant amount of work has been done on
constructing and understanding $r$-Ryser hypergraphs
(e.g.~\cite{AP, ABW, FHMW, HNSz, MSY}). For
every prime power $p$ there is a standard construction, based on the
projective plane, that gives an intersecting
$r$-Ryser hypergraph for $r=p+1$. Very recently it was proved in~\cite{ABPS} that
intersecting $r$-Ryser hypergraphs exist also for every $r=p+2$. Apart from these
infinite families, the only other values of $r$ for which the conjecture is
known to be tight are $r=7$~\cite{ABW, AP}, $r=11$~\cite{AP} and
$r=12$~\cite{FHMW}. In~\cite{ABPS}
the authors ask whether there exists a constant $K$ such that for
every $r$ there exists an intersecting $r$-partite $r$-uniform
hypergraph $\cH$ with $\tau(\cH)\geq r-K$, thus showing that Ryser's
problem is close to being best possible for every $r$. Here we answer
this question in the affirmative, by showing in particular that we may
take $K\leq 4$ for all sufficiently large
integers $r$. This is our main result
(Theorem~\ref{main}), and it appears in Section~\ref{mainthm}. We also
give a new construction for intersecting $r$-Ryser hypergraphs for special
values of $r$ in Section~\ref{ee}.  

\section{Basic construction}

As in most known constructions for this problem (for example~\cite{ABPS}), our construction will be based on finite projective
planes and the corresponding affine planes. Recall that a {\it
  projective plane} 
of order $p$ is a $(p+1)$-uniform hypergraph with $p^2+p+1$ vertices,
with the property that each pair of edges (called {\it lines})
intersects in exactly one vertex, and each pair of vertices (called
{\it points}) is contained in exactly one line. It is well-known that
for every prime power $p$, there exists a projective plane PG(2,$p$)
of order $p$.

The affine plane AG$(2,p)$ is 
constructed from PG(2,$p$) by deleting a line $L$ and all of its
points (so we delete exactly one point from each line $L'\not=L$). 
Thus AG$(2,p)$ has $p^2$ points. Its lines each have $p$
points, and they fall into $p+1$ parallel classes $B_i$, each of which
is a set of $p$ disjoint lines (corresponding to the lines of
PG(2,$p$) passing through a single point of $L$). Any two
lines from different parallel classes have exactly one vertex in
common. 

We
define a hypergraph $\cA_p$ constructed from
AG$(2,p)$ by choosing an arbitrary point $x$ and removing it, together
with all lines that contain $x$. The remaining lines of AG$(2,p)$ form
the edges of $\cA_p$.

\begin{proposition} \label{Aprops} The hypergraph $\cA_p$ has the following properties.
\begin{enumerate}
\item $\cA_p$ is $(p+1)$-partite, with vertex classes $V_1,\ldots,
  V_{p+1}$ where $\{x\}\cup V_1,\ldots,\{x\}\cup V_{p+1}$ are the lines
  of AG$(2,p)$ containing $x$. Note $|V_i|=p-1$ for each $i$.
\item $\cA_p$ is $p$-uniform.
\item The edges of $\cA_p$ fall into $p+1$ parallel classes $C_i$,
  each of which is a set of $p-1$ disjoint edges. Any two
edges from different parallel classes have exactly one vertex in common.
\item Each edge of $C_i$ is disjoint from $V_i$.
\end{enumerate}
\end{proposition}

Our aim is to construct an intersecting hypergraph based on $\cA_p$,
by adding a gadget for each parallel class $C_i$ to make it
intersect. To show that the cover number of the resulting construction
is large  we will make use of the following theorem of Jamison \cite{J}
and Brouwer and Schrijver~\cite{BS}. 

\begin{theorem}\label{BrSch}
$\tau({\rm{AG}}(2,p))=2p-1.$
\end{theorem}

Note that a cover of size $2p-1$ can be obtained by choosing a
parallel class $C_i$, and taking all points from one line in $C_i$ and
one point from each of the remaining lines in $C_i$.

\section{Near-extremal constructions}\label{ne}

Let $\cJ$ be an $r_0$-partite $r_0$-uniform intersecting hypergraph
with $r_0\leq p$ and $\tau(\cJ)\geq r_0-1-d_0\geq 2$ for some $d_0\geq 0$. 

Set $r=p+r_0$. We construct an $r$-partite $r$-uniform hypergraph
$\cH_r$ as follows. Fix a copy of $\cA_p$ with vertex classes
$V_1,\ldots,V_{p+1}$ (as in Proposition \ref{Aprops}). For each parallel class $C_i$ of $\cA_p$ place a
copy $\cJ^i$ of $\cJ$ with one vertex class in $V_i$ and the remaining
$r_0-1$ classes in $V_{p+2}\ldots V_{p+r_0}$ in an arbitrary
way, such that all $\cJ^i$ are disjoint from each other and from
$\cA_p$. Extend every edge $e$ of $C_i$ to $|\cJ|$ edges 
$e\cup f$ of $\cH_r$ by
appending each edge $f$ of $\cJ^i$ to $e$. 
Thus the edge set of $\cH_r$ is $\bigcup_{i=1}^{p+1}\{e\cup f:e\in C_i,f\in\cJ^i\}$.

\begin{proposition}\label{Hprops} The hypergraph $\cH_r$ has the following properties.
\begin{enumerate}
\item $\cH_r$ is an $r$-partite $r$-uniform intersecting hypergraph,
\item $\tau(\cH_r)\geq r-1-(d_0+1)$.
\end{enumerate}
\end{proposition}
\begin{proof}
The definitions, together with Part 3 of Proposition~\ref{Aprops}, imply that $\cH_r$ is $r$-partite and
$r$-uniform. To see that $\cH_r$ is intersecting, let $e\cup f$ and
$e'\cup f'$ be
two edges of $\cH_r$. If $e=e'$ or if $e$ and $e'$ are from different parallel
classes of $\cA_p$ then they intersect in $\cA_p$, implying that
$e\cup f$ and $e'\cup f'$ intersect in $\cH_r$. If $e$ and $e'$ are from
the same parallel class of $\cA_p$ then $f$ and $f'$ are
two (not necessarily distinct) edges from the same copy of the
intersecting hypergraph 
$\cJ$, and therefore they intersect. 

To estimate $\tau(\cH_r)$, consider a minumum cover $T$. 

\medskip

\noindent{\bf Case 1:} No vertex of $T$ is in any $\cJ^i$.

In this case $T\cup\{x\}$ is a cover of the affine plane AG$(2,p)$,
which by Theorem~\ref{BrSch} must have size at least $2p-1$. Hence
$|T|\geq 2p-2\geq p+r_0-2=r-2\geq r-1-(d_0+1)$.

\medskip

To address the remaining cases, we claim that if $T$ contains a vertex
$z$ of
$\cJ^i$ then $T$ contains a cover of $\cJ^i$. To see this, suppose
on the contrary that some edge $f$ of $\cJ^i$ is disjoint from
$T$. Since $T$ is a minimum cover there exists an edge $e\cup f'$ of
$\cH_r$ such that $T\cap(e\cup f')=\{z\}$, where $e\in C_i$ and
$f'\in\cJ^i$. Then $T\cap e=\emptyset$. But then $e\cup f\in\cH_r$ is
disjoint from $T$, contradicting the fact that $T$ is a cover. This
verifies the claim. 

\medskip

\noindent{\bf Case 2:} For some $i\not= j$, the cover $T$ intersects
$\cJ^i$ but not $\cJ^j$.

Then by the claim $T$ contains a cover of $\cJ^i$, which has size at
least $r_0-1-d_0$. Since $T$ has no vertices in $\cJ^j$ it must
cover $C_i$ within the vertex set of $\cA_p$ which is disjoint from
$\cJ^i$. Since the $p-1$ edges in $C_i$ are disjoint, we get another $p-1$ vertices in $T$, for a total of
$p-1+r_0-1-d_0=r-1-(d_0+1)$ as required.

\medskip

\noindent{\bf Case 3:} $T$ intersects $\cJ^i$ for every $i$.

Since the $\cJ^i$ are all disjoint we find by the claim that
$|T|\geq\tau(\cJ)(p+1)\geq 2p+2>r-1-(d_0+1)$.    

\medskip

Therefore in all cases the statement holds.
\end{proof}

\section{The main theorem}\label{mainthm}

We begin with a construction of $r$-partite $r$-uniform hypergraphs when $r$ has a special form.

\begin{lemma}\label{kprimes}
Let $r=\sum_{i=1}^kp_i+1$, where each $p_i$ is a prime power
and $p_i\geq\sum_{j<i}p_j+1$ for each $i\geq 2$. Then there
exists an $r$-partite $r$-uniform intersecting hypergraph $\cH_r$ with
$\tau(\cH)\geq r-k$.
\end{lemma}

\begin{proof}
We use induction on $k$. The case $k=1$ is dealt with by the standard
example of the truncated projective plane (formed by removing one
point from the projective plane, together with every line containing
it): we obtain an example with $r=p+1$ classes and $\tau=p=r-1\ge2$.

Assume $k\geq 2$ and let $\cH_s$ be a hypergraph with the claimed
properties for $s=r-p_k$. Observe that the conditions guarantee
$p_k\geq s$. Note also that $\tau(\cH_s)\geq 2$. Construct
$\cH_r$ as in Section~\ref{ne},
starting with the hypergraph $\cA_{p_k}$ and using
$\cJ=\cH_s$. Then by the induction hypothesis
$\tau(\cJ)\geq s-(k-1)=s-1-d_0$ where $d_0= k-2$. By
Proposition~\ref{Hprops} we obtain an intersecting  hypergraph $\cH_r$
for $r=s+p_k$ that is $r$-partite and $r$-uniform, that satisfies 
$$\tau(\cH_r)\geq r-1-(d_0+1)=r-k.$$
This completes the proof.
\end{proof}

In fact we will use this lemma below only when each $p_i$ is prime and $k=3$.

To show the existence of suitable primes we use the following
classical result of Montgomery and Vaughan~\cite{MV}.

\begin{theorem}\label{MontV}
There exist $Q$, $\gamma>0$ and $N$ such that for all $n>N$, all but at
most $Qn^{1-\gamma}$ even integers in the interval $(0,n)$ are the sum
of two primes.
\end{theorem}

We are now ready to prove our main theorem.

\begin{theorem}\label{main}
There exists $M$ such that for every integer $r>M$
\begin{itemize}
\item if $r$ is even then there exists an $r$-partite $r$-uniform
  intersecting hypergraph $\cH$ with $\tau(\cH)\geq r-3$,
\item  if $r$ is odd then there exists an $r$-partite $r$-uniform
  intersecting hypergraph $\cH$ with $\tau(\cH)\geq r-4$.
\end{itemize}
\end{theorem}

\begin{proof}
We note that the second claim follows immediately from the
first, since we may construct an 
$r$-partite $r$-uniform
  intersecting hypergraph $\cH'$ from an $(r-1)$-partite $(r-1)$-uniform
  intersecting hypergraph $\cH$ by adding a new vertex class, and adding a
  new vertex in this class to every edge of $\cH$. Then $\tau(\cH')=\tau(\cH)$. 
Thus we may assume that $r$ is even.

Our aim is to write $r=p_1+p_2+p_3+1$,
where $p_1, p_2, p_3$ satisfy $p_2>p_1$ and $p_3>p_2+p_1$, as in Lemma~\ref{kprimes}.
Let $Q$, $N$ and $\gamma$ be as in
Theorem~\ref{MontV}. For an interval $I$ we write $pI$ for
the number of primes in $I$, and for a real number $x$ we let $p(x)$ denote
$p[1,x]$. The Prime Number Theorem tells us that $p(x)=(1+o(1))x/{\log
  x}$. Therefore there exists $M\geq N$ such that for all $t\geq M$ we have 
$$p({3t}/4)+p( t/8)-p( t/2)-p( t/4)-Q\cdot(t/2)^{1-\gamma}\geq 1.$$

Let $r>M$ be an even integer. Set $t=r-1$. Let $w=p(\frac
t2,\frac{3t}4]=p(\frac{3t}4)-p(\frac t2)$, so there are $w$ choices for
a prime $p_3$ in the interval $(\frac t2,\frac{3t}4]$. Thus there are
  $w$ integers in the interval $[\frac t4,\frac t2)$ of the form
    $t-p_3$ where $p_3$ is prime.

Now we show that one of these $w$ integers can be written as $p_1+p_2$
for distinct primes $p_1$ and $p_2$. Let us call such an integer {\it
  good}. By Theorem~\ref{MontV} there are 
at most $z=Q\cdot (\frac t2)^{1-\gamma}$ integers in $[\frac t4,\frac t2)$
  that are not the sum of two primes. The number $y$ of integers in
  $[\frac t4,\frac t2)$ of the form $2p$ where $p$ is prime is
    $p[\frac t8,\frac t4)=p(\frac t8, \frac t4]$ since neither  $t/4$
    nor $t/2$ is an integer. Thus $y=p(\frac t4)-p(\frac t8)$. Thus
    the number of good integers is at least
$$w-z-y=p({3t}/4)+p( t/8)-p( t/2)-p( t/4)-Q\cdot(t/2)^{1-\gamma}\geq 1.$$

Therefore a good integer exists and we can write $r-1=t=p_1+p_2+p_3$
where $p_1<p_2$ and $p_3\geq \frac{t+1}2=
1+\frac{t-1}2\geq1+p_1+p_2$. Therefore by Lemma~\ref{kprimes} there
exists an $r$-partite $r$-uniform intersecting hypergraph $\cH_r$ with
$\tau(\cH)\geq r-3$ as required.
\end{proof}

\section{Extremal constructions}\label{ee}

Here we give another construction based on the hypergraph $\cA_p$ of
an $r$-partite $r$-uniform intersecting hypergraph with cover number
exactly $r-1$. It exists whenever $r=2p-1$ and both $p$ and $p-1$ are prime powers.

Our construction gives a tight example for Ryser's conjecture for a few previously unknown values of $r$.  Note that if $p$, $p-1$ are both prime powers then
one of $p$, $p-1$ must be a power of 2.  If $p=2^{i-1}+1$ then
$r=2^i+1$: since $r-1$ is a prime power, there is already an extremal
construction for this $r$.  However,  
if $p=2^{i-1}$ and $p-1$ is also prime, then we obtain a construction for $r=2p-1=2^i-1$.  The construction gives a previously unknown value of $r$ if 
neither of  $r-1=2^i-2$ and $r-2=2^i-3$ is a prime power.
For example, this holds when $i$ is any of 
8, 18, 32, 62, 90, 108, 128, 522, 608, 1280, 2204, 2282, 3218, 4254, 4424, 9690, 9942, 11214, 19938.   We remark that
the examples on this list all satisfy $r=2p-1$ where $p-1=2^{i-1}-1$ is a Mersenne prime (and recall that it is unknown whether
there are infinitely many Mersenne primes).

We now describe the construction.  We repeat the general idea of
Section~\ref{ne}. Let $p$ be a 
prime power such that $p-1$ is also a prime power. This time we
begin with the hypergraph $\cJ$ formed from AG$(2,p-1)$ by removing
the lines of one parallel class $B_1$ and declaring them to be the
classes of a vertex partition into $p-1$ vertex classes, each of size
$p-1$. Then $\cJ$ is 
an $r_0$-partite $r_0$-uniform hypergraph with $r_0= p-1$, with $p-1$
parallel classes of edges, each containing $p-1$ vertices, and any two
edges from different parallel classes intersect. 

Set $r=p+r_0=2p-1$. We construct an $r$-partite $r$-uniform hypergraph
$\cG_r$ as follows. Fix a copy of $\cA_p$ with vertex classes
$V_1,\ldots,V_{p+1}$. For each parallel class $C_i$ of $\cA_p$ place a
copy $\cJ^i$ of $\cJ$ with one vertex class in $V_i$ and the remaining
$r_0-1$ classes in $V_{p+2}\ldots V_{p+r_0}$ in an arbitrary
way, such that all $\cJ^i$ are disjoint from each other and from
$\cA_p$. Take an arbitrary matching between the set $C_i$ and the set
of parallel classes of $\cJ^i$ and extend every edge $e$ of $C_i$ to
$p-1$ edges  $e\cup f$ of $\cG_r$ by
appending to $e$ each edge $f$ of the parallel class of $\cJ$ matched to $e$.

\begin{theorem}\label{Gprops} The hypergraph $\cG_r$ has the following properties.
\begin{enumerate}
\item $\cG_r$ is an $r$-partite $r$-uniform intersecting hypergraph,
\item $\tau(\cG_r)\geq r-1$.
\end{enumerate}
\end{theorem}
\begin{proof}
It follows immediately from the definitions that $\cG_r$ is $r$-partite and
$r$-uniform. To see that $\cG_r$ is intersecting, let $e\cup f$ and
$e'\cup f'$ be
two edges of $\cG_r$. If $e=e'$ or if $e$ and $e'$ are from different parallel
classes of $\cA_p$ then they intersect in $\cA_p$, implying that
$e\cup f$ and $e'\cup f'$ intersect in $\cG_r$. If $e$ and $e'$ are from
the same parallel class of $\cA_p$ then $f$ and $f'$ are
two edges from distinct parallel classes of $\cJ$, and therefore they
intersect.  

To estimate $\tau(\cG_r)$, consider a minimum cover $T$. 
If no vertex of $T$ is in any $\cJ^i$, then $T\cup\{x\}$ is a cover of
the affine plane AG$(2,p)$, 
where $x$ is the vertex deleted from AG$(2,p)$ in the construction of $\cA_p$.
By Theorem~\ref{BrSch}, this must have size at least $2p-1$. Hence
$|T|\geq 2p-2=r-1$.

To conclude the proof we show that there exists a minimum cover $T$
that is disjoint from all $\cJ^i$.
To see this, suppose that $T$ contains a vertex $z$ of $\cJ^i$.
Since $T$ is a minimum cover there exists an edge $e\cup f'$ of
$\cG_r$ such that $T\cap(e\cup f')=\{z\}$, where $e\in C_i$ and
$f'\in\cJ^i$. Then $T\cap e=\emptyset$. But the $p-1$ edges $e\cup g$
for all $g$ in a parallel class of $\cJ^i$ are edges of $\cG_r$, and
therefore $T$ contains $p-1$ vertices of $\cJ^i$ to cover them. But
then the edges of $\cG_r$ meeting $\cJ^i$ could instead be covered by 
$p-1$ vertices of $\cA_p$, one from each edge in $C_i$. Repeating this
argument shows the existence of a minimum $T$ disjoint from all
$\cJ^i$, thus completing the proof.
\end{proof}
 
We remark in closing that except for a few sporadic small examples,
all constructions of intersecting $r$-partite hypergraphs $\cH$ with
$\tau(\cH)$ close to $r$ seem to be based in some way on finite
projective planes, and hence depend on the existence of these special
structures. It would be interesting either to find a different type of
construction, or to show that near-extremal constructions must contain
large pieces from a projective plane. 
\bigskip

\noindent{\bf Acknowledgement.} We would like to thank an anonymous
referee for a careful reading and helpful comments.


\begin{thebibliography}{00}

\bibitem{AP} A. Abu-Khazneh and A. Pokrovskiy, Intersecting extremal
  constructions in Ryser's Conjecture for $r$-partite hypergraphs,
  {\em J. Combin. Math. Combin. Comput.}, to appear

\bibitem{ABPS} A. Abu-Khazneh, J. Bar\'at, A. Pokrovskiy, T. Szab\'o, A
  family of extremal hypergraphs for Ryser's Conjecture, arXiv:1605.06361

\bibitem{A} R. Aharoni, Ryser's Conjecture for tripartite 3-graphs, {\em Combinatorica} {\bf 21} (2001), 1--4.

\bibitem{ABW} R. Aharoni, J. Bar\'at, I. Wanless, Multipartite hypergraphs achieving equality in Ryser's Conjecture, {\em Graphs Combin.} {\bf 32} (2016), 1--15.

\bibitem{BS} A.E. Brouwer and A. Schrijver, The blocking number of an
  affine space, {\em J. Comb. Theory Series A} {\bf 24} (1978), 251--253.

\bibitem{FHMW} N. Franceti\'c, S. Herke, B.D. McKay, I. Wanless, On
  Ryser's Conjecture for linear intersecting multipartite hypergraphs,
  {\em European J. Comb.} {\bf 61} (2017), 91--105.

\bibitem{HNSz} P. Haxell, L. Narins, T. Szab\'o, Extremal hypergraphs
  for Ryser's Conjecture, {\em J. Comb. Theory Series A}, to appear

\bibitem{HS} P. Haxell and A. Scott, On Ryser's Conjecture, {\em Elec. J. Comb} {\bf 19} (2012), P23, 10 pp.

\bibitem{J} R.E. Jamison, Covering finite fields with cosets of subspaces, {\em J. Comb. Theory Series A}
{\bf 22} (1977), 253--266.

\bibitem{MSY} T. Mansour, C. Song, R. Yuster, A comment on Ryser's Conjecture for intersecting hypergraphs, {\em Graphs Combin.} {\bf 25} (2009), 101--109.

\bibitem{MV} H.L. Montgomery and R.C. Vaughan, 
The exceptional set in Goldbach's problem, {\em Acta Arith.}
{\bf 27} (1975), 353--370.

\bibitem{R} H.J. Ryser, Neuere Probleme der Kombinatorik, in {\em
  Vortrage \"uber Kombinatorik Oberwolfach, Mathematisches
  Forschungsinstitut Oberwolfach} (1967), 24--29.

\bibitem{T} Zs. Tuza, On the order of vertex sets meeting all edges of a 3-partite hypergraph, {\em Ars Comb.} {\bf 24(A)} (1987), 59--63.

\end{thebibliography}
\end{document}